\documentclass[12pt, reqno, a4paper]{amsart}
\usepackage{amsfonts, amsmath, amsthm, amstext, amssymb}
\usepackage{mathtools}
\usepackage{url}
\usepackage{paralist}
\usepackage{tikz}\usetikzlibrary{arrows}
\usepackage{tikz-cd}
\usepackage[english]{babel}
\usepackage{framed}
\usepackage{amsrefs}
\usepackage[marginratio=1:1,height=660pt,width=410pt,tmargin=100pt]{geometry}
\usepackage{lipsum}
\usepackage{layout}
\usepackage{hyperref}
\usepackage[shortlabels]{enumitem}

\DeclareFontFamily{U}{mathx}{\hyphenchar\font45}
\DeclareFontShape{U}{mathx}{m}{n}{
      <5> <6> <7> <8> <9> <10>
      <10.95> <12> <14.4> <17.28> <20.74> <24.88>
      mathx10
      }{}
\DeclareSymbolFont{mathx}{U}{mathx}{m}{n}
\DeclareFontSubstitution{U}{mathx}{m}{n}
\DeclareMathAccent{\widecheck}{0}{mathx}{"71}

\makeatletter
\@namedef{subjclassname@2020}{\textup{2020} Mathematics Subject Classification}
\makeatother

\usepackage[OT2,T1]{fontenc}
\DeclareSymbolFont{cyrletters}{OT2}{wncyr}{m}{n}
\DeclareMathSymbol{\Sha}{\mathalpha}{cyrletters}{"58}

\DeclareMathOperator{\tbT} {\bf{T}}
\DeclareMathOperator{\tbL} {\bf{L}}

\DeclareMathOperator{\Aut} {Aut}
\DeclareMathOperator{\End} {End}
\DeclareMathOperator{\Gal} {Gal}
\DeclareMathOperator{\Disc} {Disc}
\DeclareMathOperator{\Spec} {Spec}
\DeclareMathOperator{\reg} {reg}

\DeclareMathOperator{\topo} {top}

\DeclareMathOperator{\et} {\text{\'{e}t}}

\DeclareMathOperator{\CC} {\mathbb{C}}

\DeclareMathOperator{\RR} {\mathbb{R}}

\DeclareMathOperator{\QQ} {\mathbb{Q}}

\DeclareMathOperator{\ZZ} {\mathbb{Z}}

\newtheorem{lemma}{Lemma}[section]
\newtheorem{thm}[lemma]{Theorem}

\theoremstyle{definition}
\theoremstyle{definition}
\theoremstyle{definition}
\theoremstyle{definition}

\newcommand{\mb}[1]{\mathbb{#1}}
\newcommand{\mc}[1]{\mathcal{#1}}

\newcommand{\ra}{\rightarrow}
\newcommand{\ol}[1]{\overline{#1}}
\newcommand{\wt}[1]{\widetilde{#1}}
\newcommand{\tb}[1]{\textbf{#1}}
\newcommand{\bs}{\backslash}

\begin{document}
\title{Arithmetic sparsity in mixed Hodge settings}
\author{Kenneth Chung Tak Chiu}
\address{Department of Mathematics,  University of Toronto, 40 St. George Street, Ontario, Canada.}
\email{kennethct.chiu@alumni.utoronto.ca}
\begin{abstract}
Let  $X$ be a smooth irreducible quasi-projective algebraic variety over a number field $K$. 
Suppose $X$ is equipped with a $p$-adic \'{e}tale local system compatible with an admissible graded-polarized variation of mixed Hodge structures on the complex analytification of $X_{\CC}$.
We prove that the $S$-integral points in $X$ are covered by subpolynomially many geometrically irreducible $K$-subvarieties, each lying in a fiber of the mixed period mapping arising from the variation of mixed Hodge structures. This is based on recent works by Brunebarbe-Maculan and Ellenberg-Lawrence-Venkatesh. As an application, we prove that there are subpolynomially many $S$-integral Laurent polynomials with fixed reflexive Newton polyhedron $\Delta$ and fixed non-zero principal $\Delta$-determinant.  Our results answer a question asked by Ellenberg-Lawrence-Venkatesh.
\end{abstract}

\maketitle

\section{Introduction}\label{introduction}
Faltings theorem \cite{Fal83} states that curves of genus $\geq 2$ over any number field $K$ have only finitely many rational points. In higher dimensions, Bombieri-Lang conjecture states that rational points of a smooth projective variety of general type defined over $K$  are not Zariski-dense, in other words, finitely many irreducible algebraic $K$-subvarieties properly contained in the variety are enough to cover these rational points.
One could also look at statements for certain quasi-projective varieties, e.g. moduli spaces of varieties with fixed set of properties, although rational points are replaced by $S$-integral points, where $S$ is a finite set of places of $K$.
There were affirmative answers to these questions in the cases for moduli spaces of abelian varieties of fixed dimension \cite{Fal83}, curves of fixed genus \cite{Fal83}, hypersurfaces of fixed large degree \cite{LV20}, and smooth hypersurfaces representing an ample class in the Neron-Severi group of an abelian variety \cite{LS20}.
We aim at proving statements where non-density is weakened to sparsity, i.e.  subpolynomial growth rate, in terms of the heights of the $S$-integral points in a dense open subset.

\subsection{Arithmetic sparsity in mixed Hodge settings}\label{subsection:Arithmetic sparsity in mixed Hodge settings}
Let $K$ be a finite extension of $\QQ$ with an embedding into $\CC$. 
Let  $X$ be a smooth irreducible quasi-projective algebraic variety over $K$ with an embedding into the projective space $\mb{P}^m_K$ for some positive integer $m$. 
Let $S$ be a finite set of places of $K$ such that $X$ has a smooth integral model over the ring $\mc{O}_{K,S}$ of $S$-integers. 
Our goal is to prove the following main theorem, which is a mixed Hodge analogue of \cite[Theorem 1.2]{ELV23}:

\begin{thm}\label{Sparsity, surjective quasi-projective}
Let $\pi:V\ra X$ be a surjective quasi-projective morphism over $K$ from an irreducible algebraic variety $V$. There exists a non-empty Zariski open subset $X^*$  of $X$ such that for any $i$, the higher direct image with compact support $(R^i(\pi|_{V_{\CC}})_! \QQ_{V_{\CC}})|_{X^*_{\CC}}$ underlies an admissible graded-polarized variation of $\QQ$-mixed Hodge structures, and such that for any $\varepsilon>0$, the $S$-integral points of $X^*$ with height at most $B$ are covered by $O_{\varepsilon}(B^{\varepsilon})$ geometrically irreducible $K$-subvarieties, each lying in a single fiber of the mixed period mapping  $\Phi$ arising from the variation.
\end{thm}

The statement with $K$ replaced by $\CC$ and without the latter condition about integral points is proved by Brosnan-El Zein  \cite[Cor. 8.1.22]{BE} and Fujino-Fujisawa \cite[Theorem 4.13]{FF}.
Due to Theorem \ref{Sparsity, surjective quasi-projective}, to count points in $X$, we only have to count points in each fiber of a period mapping. This will become useful when the period mapping is not a constant map, i.e. the variation of mixed Hodge structure is not trivial, say in situations where the infinitesimal Torelli theorems hold.

Theorem \ref{Sparsity, surjective quasi-projective} will be proved using the following analogous theorem for variations of mixed Hodge structures with compatible $p$-adic \'{e}tale local system:

\begin{thm}\label{Sparsity, abstract VMHS}
Let $\mc{L}_{an}$ be a local system of $\ZZ$-modules on the analytification $X^{an}_{\CC}$.  Suppose $\mc{L}_{an}$ underlies an admissible graded-polarized variation of $\QQ$-mixed Hodge structures (VMHS). Let $p$ be a prime for which $\mc{L}_{an}$ is $p$-torsion-free. Suppose for each positive integer $n$, there exists on $X_K$ an \'{e}tale local system $\mc{L}_{n, \et}$ of $\ZZ/p^n\ZZ$-modules such that $(\mc{L}_{n,\et,\CC})_{an}\simeq \mc{L}_{an}\otimes (\ZZ/p^n\ZZ)$, where $\mc{L}_{n,\et,\CC}$ is the pullback of $\mc{L}_{n,\et}$ to $X_{\CC}$. Then for any $\varepsilon>0$, the $S$-integral points of $X$ with height at most $B$ are covered by $O_{\varepsilon}(B^{\varepsilon})$ geometrically irreducible $K$-subvarieties, each lying in a single fiber of the mixed period mapping $\Phi$ arising from the variation.
\end{thm}

The techniques of the proof of Theorem \ref{Sparsity, abstract VMHS} will be based on the recent papers by Brunebarbe-Maculan \cite{BM22} and Ellenberg-Lawrence-Venkatesh \cite{ELV23}: the geometric portion of the argument is to construct covers of $X$ under which the preimages of certain subvarieties have large degrees, by using a result in \cite{BM22} about small degree normal cycles; the arithmetic portion of the argument is to apply Broberg's theorem \cite{Bro04} (which builds on fundamental ideas of Bombieri-Pila \cite{BP89} and Heath-Brown \cite{Hea02}) about the number of divisors of bounded degrees required to contain the rational points of bounded height of an arbitrary irreducible closed subvariety of fixed dimension and degree. 
We remark that if the mixed period mapping arising from the VMHS is quasi-finite, then $\mc{L}_{an}$ is a large local system (i.e. the pullback of $\mc{L}_{an}$ by any non-constant morphism from a normal irreducible complex variety has infinite monodromy) and Theorem \ref{Sparsity, abstract VMHS} follows from the main result of the paper by Brunebarbe-Maculan \cite{BM22}.

\subsection{Sparsity of $S$-integral Laurent polynomials with fixed data}
We discuss an application of Theorem \ref{Sparsity, surjective quasi-projective}.
Let $n$ be a positive integer. Let $\tbL$ be the Laurent polynomial ring $\mb{C}[X_1^{\pm},\dots, X_n^{\pm}]$. Let $\tbT:= \Spec \tb{L} \cong (\CC^\times)^n$.
An $n$-dimensional convex polyhedron $\Delta$ in $\RR^n$ whose interior contains the origin is said to be \emph{integral} if all vertices of $\Delta$ belong to the lattice $\ZZ^n$. 
Such polyhedron is said to be \emph{reflexive} if its dual polyhedron 
$$\Delta^*:=\{(x_1,\dots, x_n)\in \RR^n: \sum_{i=1}^n x_iy_i\geq -1 \text{ for all } (y_1,\dots, y_n)\in \Delta\}$$
is integral.
Let $M$ be the free abelian group of rank $n$.
The Newton polyhedron $\Delta(L)$ of a Laurent polynomial $L=\sum_{m\in M} a_m X^m \in \tb{L}$ is the convex hull of integral points $m\in M$ such that $a_m\neq 0$.
Let $\tbL(\Delta)$ be the space of all Laurent polynomials with the Newton polyhedron $\Delta$.

Every Laurent polynomial $L$ defines the affine hypersurface
$$Z_L:=\{X\in \tbT: L(X)=0\}.$$
For any $L= \sum_{m\in M} a_m X^m\in \tbL(\Delta)$ and any $l$-dimensional face $\Delta'$ of $\Delta$, let 
$$L^{\Delta'}(X):= \sum_{m\in\Delta'} a_mX^{m}.$$
A Laurent polynomial $L\in \tbL(\Delta)$ and its affine hypersurface $Z_L$  are said to be \emph{$\Delta$-regular} if for every $l$-dimensional face $\Delta'\subset \Delta$ $(l>0)$, the polynomial equations
$$L^{\Delta'}(X)=X_1\frac{\partial}{\partial X_1} L^{\Delta'}(X)=\cdots =X_n\frac{\partial}{\partial X_n} L^{\Delta'}(X)=0.$$
have no common solutions in $\tbT$.

Gelfand, Kapranov, and Zelevinski introduced the \emph{principal $\Delta$-determinant} $\Disc_{\Delta}(L)$ of  a Laurent polynomial $L$ with Newton polyhedron $\Delta$. It is a certain complex number attached to $L$. Since its definition is complicated to state and we are only using its properties, we refer the interested reader to their book \cite[p. 297]{GKZ94}.
A Laurent polynomial $L$ is $\Delta$-regular if and only if its principal $\Delta$-determinant $\Disc_{\Delta}(L)\neq 0$ \cite[Prop. 4.16]{Bat93}.

In Section \ref{section: Laurent polynomials}, we will use Theorem \ref{Sparsity, surjective quasi-projective} to prove the following theorem:

\begin{thm}\label{counting Laurent polynomials}
Let $S$ be a finite set of rational primes.
Let $\Delta$ be an $n$-dimensional reflexive polyhedron. 
Let $A:=\Delta \cap \ZZ^n$. 
Let $N\in \QQ^\times$.  For any $\varepsilon>0$, there are $O_{\Delta, N, S, \varepsilon}(B^{\varepsilon})$ Laurent polynomials $L=\sum_{m\in A} a_m X^m$ with Newton polyhedron $\Delta$ and principal $\Delta$-determinant $N$, and with $S$-integral coefficients such that 
$$\max_{m\in A, v\in S\cup\{\infty\}} |a_m|_v\leq B.$$
\end{thm}

It was shown by Batyrev that reflexivity of the polyhedron is a necessary and sufficient condition for the characterization of $\Delta$-regular affine hypersurfaces in tori that have Calabi-Yau projective closure in the toric variety with only canonical singularities \cite[Theorem 12.2]{Bat93}. 
Properties of principal $\Delta$-determinant and Batyrev's infinitesimal Torelli theorem \cite{Bat93} will be used to prove Theorem \ref{counting Laurent polynomials}.

The question of what happens in the mixed Hodge setting was asked by Ellenberg-Lawrence-Venkatesh \cite[p. 4]{ELV23}.
They deduced from their main theorem that there are $O_{\varepsilon}(B^{\varepsilon})$ regular $S$-integral homogeneous polynomials with fixed number of variables,  degree $d\geq 3$, and discriminant, up to the action by an arithmetic group.
The condition on the degree in their sparsity result is relaxed in comparison with the non-density result in Lawrence-Venkatesh \cite{LV20}.
It would be interesting to explore the connection between Theorem \ref{counting Laurent polynomials} and this  sparsity result on homogeneous polynomials by Ellenberg-Lawrence-Venkatesh \cite{ELV23},
but it is not obvious because our polyhedron $\Delta$ is assumed to contain the origin.
It is also worth mentioning that in our case for Laurent polynomials, we do not have to count up to the action by a group.

\subsection{Acknowledgements}
I would like to thank Jacob Tsimerman for helpful discussions.
I would also like to thank the referees for their valuable suggestions.
One of the referees pointed out the use of Saito's theory which simplifies the deduction of Theorem \ref{Sparsity, surjective quasi-projective} from Theorem \ref{Sparsity, abstract VMHS}.

\section{Reduction of Theorem \ref{counting Laurent polynomials} to Theorem \ref{Sparsity, surjective quasi-projective}}\label{section: Laurent polynomials}
Let $S$ be a finite set of rational primes.
Let $\Delta$ be an $n$-dimensional reflexive polyhedron.  Let $A:=\Delta \cap \ZZ^n$. 
Let $\CC^A$ be the affine complex space parametrizing coefficients of Laurent polynomials 
$L=\sum_{m\in A} a_mX^m.$
The multiplicative group $\CC^\times$ acts on $\CC^A$ by componentwise multiplication.
For any $m\in A$, write $m=(m_1,\dots, m_n)$.
We define an action of the complex tori $\tbT:=(\CC^\times)^n$ on $\mb{C}^A$ as follows: for any $\tb{t}\in \tb{T}$ and $(a_m)_{m\in A} \in\mb{C}^{A}$, we set 
$\tb{t}\cdot (a_m)_{m\in A}:=(t_1^{m_1}\cdots t_n^{m_n}a_m)_{m\in A}$, i.e. the $|A|$-tuple of coefficients of $X^m$ in 
$$\sum_{m\in A}a_m  (\tb{t}X)^m.$$
The $\CC^\times$-action and the $\tbT$-action preserve $\Delta$-regularity \cite[Prop. 11.2 or Prop. 4.6]{Bat93}.

\subsection{Lemmas on $S$-integral Laurent polynomials}

\begin{lemma}\label{vertex, coefficient count}
Let $N\in \QQ^\times$.
Suppose $L=\sum_{m\in A} a_m X^m$ is a Laurent polynomial with principal $\Delta$-determinant $N$ and $S$-integral coefficients such that 
$$\max_{m\in A, v\in S\cup\{\infty\}} |a_m|_v\leq B.$$
Let $m'$ be a vertex of $\Delta$.
Then $a_{m'}$ can only attain $O_{\Delta, N, S, m'}((\log B)^{|S|})$ possible values.
\end{lemma}

\begin{proof}
By \cite[Cor. 2.5, p. 318]{GKZ94}, since $m'$ is a vertex of $\Delta$, we have $\Disc_{\Delta}(L)=c\cdot a_{m'}^k\cdot h$, where $c\in \QQ$, $k$ is some non-negative integer, and $h$ is the value at $(a_m)_{m\in A}$ of some polynomial $\eta$ in the ring $\ZZ[A]$. Here $c$ depends only on $\Delta$, while  $k$ and $\eta$ depend only on $\Delta$ and $m'$. Write $c=r/q$, $a_{m'}=r'/q'$, $h=r''/q''$, $a_m=r_m/q_m$, and $N=r_N/q_N$, which are fractions in the lowest terms with $q, q', q'', q_m, q_N>0$.
We have
$$\frac{r}{q}\cdot \left(\frac{r'}{q'}\right)^k \cdot \frac{r''}{q''}=\Disc_{\Delta}(L)=N=\frac{r_N}{q_N},$$
so $r'$ divides $r_Nqq''$.
There are only $O_{\Delta, N}(1)$ choices for the divisors of $r_N$ and $q$. 
We also know that $q''$ divides a monomial $\prod_m q_m^{f_m}$,
where the powers $f_m$ depends only on $\eta$, which in turn depends only on $\Delta$ and $m'$. 
Write $S=\{v_1,\dots, v_{\ell}\}$ and $q_m=v_1^{e_{m,1}}\cdots v_{\ell}^{e_{m,\ell}}$.
The monomial is then equal to 
$$\prod_{i=1}^{\ell} v_i^{\sum_{m\in A} f_me_{m,i}}.$$
For all $i=1,\dots, \ell$ and $m\in A$,  we have $v_i^{e_{m,i}}\leq B$, so
$$\sum_{m\in A} f_m e_{m,i}\leq \sum_{m\in A}\frac{f_m \log B}{\log v_i}=O_{\Delta, m', S}(\log B),$$
thus there are only $O_{\Delta, m', S}((\log B)^{|S|})$ choices for the divisors of $q''$.
Similarly, $q'=q_{m'}$ can only attain $O_{S}((\log B)^{|S|})$ possible values. 
Multiplying all the bounds together, the proof is completed.
\end{proof}

\begin{lemma}\label{point count in single orbit}
Let $N\in \QQ^\times$. 
Each $(\CC^\times \times \tbT)$-orbit  has 
$$O_{\Delta, N, S}((\log B)^{(n+1)|S|})$$
Laurent polynomials with Newton polyhedron $\Delta$, principal $\Delta$-determinant $N$, and $S$-integral coefficients $a_m$ such that 
$$\max_{m\in A, v\in S\cup\{\infty\}} |a_m|_v\leq B.$$ 
\end{lemma}

\begin{proof}
Let $L=\sum_{m\in A} a_m X^m$ be such Laurent polynomial in the orbit.
Let $(\alpha,\tb{t})\in \CC^\times \times \tbT$.
Suppose $(\alpha, \tb{t})\cdot L$ is again such Laurent polynomial.
Write $\alpha=r_{\alpha} e^{i\theta_{\alpha}}$ and $\tb{t}=(r_1e^{\theta_1}, \cdots, r_ne^{\theta_n})$, where $r_{\alpha}, r_1,\dots, r_n>0$ and $0\leq \theta_{\alpha}, \theta_1,\dots, \theta_n< 2\pi$. Write $m=(m_1,\dots, m_n)$.
We have
$$(\alpha, \tb{t})\cdot L
=\sum_{m\in A}\alpha a_m \tb{t}^m X^m
=\sum_{m\in A} r_{\alpha} a_m r_1^{m_1}\cdots r_n^{m_n} e^{i(\theta_{\alpha}+m_1\theta_1+\cdots+m_n\theta_n)}X^m.
$$
Since it has $S$-integral coefficients, $\theta_{\alpha}+m_1\theta_1+\cdots+m_n\theta_n=\pi$ or $0$, modulo $2\pi$, for each $m\in A$; so this sum can only attain finitely many possible values since it is bounded by $2\pi(1+|m_1|+\cdots +|m_n|)$.
Denote this finite set of possible values by $\Omega_m$, for each $m\in A$.
Since $\Delta$ is $n$-dimensional, we can pick vertices $m^{(1)},\dots, m^{(n)}\in A$ such that the $n\times n$ matrix $(m_{ij})$, where $m_{ij}=m_j^{(i)}$, has rank $n$.
Since $\Delta$ has at least $n+1$ vertices, we can pick a vertex $m^{(n+1)}$ pairwise distinct from $m^{(1)},\dots, m^{(n)}$.
Since $m^{(n+1)}$ cannot lie in the simplex $[m^{(1)},\dots, m^{(n)}]$, 
the system 
$$
\left\{
\begin{array}{l}
s_1+\cdots +s_n = 1 \\
s_1m^{(1)}+\cdots +s_n m^{(n)} = m^{(n+1)}\\
\end{array}
\right.
$$
has no solution. 
Therefore,
the $(n+1)\times (n+1)$ matrix
$$
R:=
\left(
\begin{array}{c|c}
\text{all } 1 & (m_{ij})\\
\hline
 1 & m^{(n+1)}\\
\end{array}
\right)
$$
has rank $n+1$.
For each $i=1,\dots, n+1$, since $m^{(i)}$ is a vertex,  the coefficient of the term $X^{m^{(i)}}$ in $(\alpha,\tb{t}) \cdot L$ can only attain $O_{\Delta, N, S, m^{(i)}}((\log B)^{|S|})$ possible values by Lemma \ref{vertex, coefficient count}.
Let $d_i$ be one of such possible value for each $i$. 
Since $(\alpha,\tb{t})\cdot L$ has Newton polyhedron $\Delta$, we know $d_i\neq 0$.
Since $R$ has full rank, the system
$$\log r_{\alpha}+m_1^{(i)}\log r_1+\cdots + m_n^{(i)}\log r_n = -\log |a_{m^{(i)}}|+\log |d_i|, \quad i=1,\dots, n+1.$$
has a unique solution for $(r_{\alpha}, r_1,\dots, r_n)$. 
Let $(\theta^{(1)},\dots, \theta^{(n+1)})\in \Omega_{m^{(1)}}\times \cdots \times \Omega_{m^{(n+1)}}$.
Since $R$ has full rank, the system
$$\theta_{\alpha}+m_1^{(i)}\theta_1+\cdots+m_n^{(i)}\theta_n=\theta^{(i)},\quad  i=1,\dots, n+1$$
has a unique solution for $(\theta_{\alpha}, \theta_1,\dots, \theta_n)$. 
\end{proof}

\subsection{Jacobian ideals and Jacobian rings}
We first recall the notions of Jacobian ideals and Jacobian rings of Laurent polynomials in Batyrev's paper \cite{Bat93}.
Let $S_{\Delta}$ be the subalgebra of $\tb{L}[X_0]=\CC [X_0, X_1^{\pm}, \dots, X_n^{\pm}]$ generated as a $\CC$-vector space by elements of $\CC$ and all monomials $X_0^kX_1^{m_1}\cdots X_n^{m_n}$ such that the rational point $(m_1/k,\dots, m_n/k)$  belongs to $\Delta$.
The standard grading of $\tbL[X_0]$ induces the grading of $S_{\Delta}$.
Let $S_{\Delta}^i$ be the $i$-th homogeneous component.
Let $S_{\Delta}^+$ be the maximal homogeneous ideal in $S_{\Delta}$.
For any $L\in \tb{L}$, define $L(X_0, X):= X_0L(X)-1$. 
For any $i=0,\dots, n$, 
$$L_i(X_0,X):= X_i\frac{\partial}{\partial X_i} L(X_0,X).$$
The ideal $J_{L, \Delta}$ of $S_{\Delta}$ generated by $L_0,L_1\dots, L_n$ is called the \emph{Jacobian ideal} of $L$.
The quotient ring $R_L:=S_{\Delta}/J_{L,\Delta}$ is called the \emph{Jacobian ring} of $L$.
The grading of $S_{\Delta}$ induces a grading of $R_L$.
Let $R_L^i$ be the $i$-th homogeneous component. 
Let $R_L^+$ be the maximal homogeneous ideal in $R_L$.

\subsection{Reduction of Theorem \ref{counting Laurent polynomials} to Theorem \ref{Sparsity, surjective quasi-projective}}
Let $\CC^A_{\Delta, \reg}$ be Zariski open subset of $\CC^A$ parametrizing $\Delta$-regular Laurent polynomials with Newton polyhedron $\Delta$.
For generic $L_{\Delta}=\sum_{m\in A} a_m X^m$ with Newton polyhedron $\Delta$, the principal $\Delta$-determinant $\Disc_{\Delta}(L_{\Delta})$ is a polynomial over $\QQ$ in the indeterminates $a_m$ \cite[Cor. 2.5, p. 318]{GKZ94}.
Let $Y$ be the Zariski open subset of $\QQ^A$ defined by  $\Disc_{\Delta}(L_{\Delta})\neq 0$ and $a_{m'}\neq 0$ for any vertex $m'$ of $\Delta$.
We have $Y_{\CC}=\CC^A_{\Delta,\reg}$.
Let $f:V\ra Y$ be the universal family of the affine hypersurfaces defined by these Laurent polynomials.

By results of Batyrev \cite[Theorem 8.2, Cor. 3.14]{Bat93}, the dimensions of the weight filtrations and the Hodge filtrations for the $(n-1)$-th cohomology stay the same as $L$ varies in $Y_{\CC}$. 
By computation of the Gauss-Manin connection of the universal family $f_{\CC}: V_{\CC}\ra Y_{\CC}$ by Batyrev \cite[Prop. 11.5, Theorem 11.6, Theorem 7.13]{Bat93}, the differential of the period mapping at any $L\in Y_{\CC}$ is induced by the composition $S_{\Delta}^1\ra R^1_{L}\ra \End R_L^+$, where the first map comes from quotienting $J_{L}^1:=J_{L,\Delta}\cap S^1_{\Delta}$, and the second map comes from the $R_L$-module multiplication $R^1_L\otimes R_L^+\ra R_L^+$. As $\Delta$ is reflexive, $R_L^1\ra \End R^+_L$ is injective by \cite[Theorem 12.2 (vi)]{Bat93}.
By \cite[Prop. 11.2]{Bat93}, the Jacobian ideal $J^1_L$ is isomorphic to the tangent space of the orbit 
$(\CC^\times \times \tbT)\cdot L$ at $L$.
Therefore, a tangent vector at $L$ is in the kernel of the differential of $\Phi$ if and only if it is tangent to 
$(\CC^\times \times \tbT)\cdot L$.
On one hand, $\Phi$ has zero differential at every point in $(\CC^\times \times \tbT)\cdot L$, so $\Phi$ is constant on $(\CC^\times \times \tbT)\cdot L$. 
On the other hand, dimension of a fiber of $\Phi$ is at most the dimension of the kernel of the differential of $\Phi$ at a generic point in the fiber. This dimension is in turn smaller than the dimension of a $(\CC^\times \times \tb{T})$-orbit contained in the fiber by what we have proved. 
Hence, any connected component of a fiber of $\Phi$ is a $(\CC^\times \times \tb{T})$-orbit. 

By Theorem \ref{Sparsity, surjective quasi-projective}, there exists a non-empty Zariski open subset $Y^*$ of $Y$ such that for any $\varepsilon>0$, the $S$-integral points of $Y^*$ with height at most $B$ are covered by $O_{\Delta, N,\varepsilon}(B^{\varepsilon})$ geometrically irreducible $\QQ$-subvarieties, whose collection is denoted by $\{Y_{\alpha}\}$, each lying in a single fiber of the period mapping restricted to $Y_{\CC}^*$.

Since the affine hypersurfaces in the universal family $f$ are smooth, and since we are looking at the middle cohomology, the sheaf $(R^{n-1}(f|_{V_{\CC}})_! \QQ_{V_{\CC}})|_{Y^*_{\CC}}$ is dual to $(R^{n-1}(f|_{V_{\CC}})_* \QQ_{V_{\CC}})|_{Y^*_{\CC}}$ by \cite[Lemma-Def. 6.25, Cor. 6.26]{PS}. Hence, we can replace the period mapping attached to higher direct image with compact support in the previous paragraph by the period mapping attached to the usual higher direct image.

Since each $Y_{\alpha}$ is geometrically irreducible and is contained in a fiber of the period mapping, it is contained in a  $(\CC^\times \times \tbT)$-orbit.
Therefore, the $S$-integral points of $Y^*$ with height at most $B$ are covered by $O_{\Delta, N, \varepsilon}(B^\varepsilon)$ $(\CC^\times \times \tbT)$-orbits. 

By applying Theorem \ref{Sparsity, surjective quasi-projective} again, with irreducible components of $Y\bs Y^*$ instead of $Y$; and repeat with irreducible subvarieties of smaller and smaller dimensions, we know that the $\Delta$-regular Laurent polynomials with Newton polyhedron $\Delta$ and $S$-integral coefficients of heights at most $B$, are in $O_{\Delta, N, \varepsilon}(B^\varepsilon)$ $(\CC^\times \times \tbT)$-orbits. By Lemma \ref{point count in single orbit}, each orbit has $O_{\Delta, N, S, \varepsilon}(B^\varepsilon)$ such Laurent polynomials with principal $\Delta$-determinant $N$.  Theorem \ref{counting Laurent polynomials} follows.

\section{Proofs of Theorem \ref{Sparsity, surjective quasi-projective} and Theorem \ref{Sparsity, abstract VMHS}}

\subsection{Proof of Theorem \ref{Sparsity, abstract VMHS}}
The techniques are based on \cite{BM22} and \cite{ELV23}.
Let $\ol{X}$ be the Zariski closure of $X$ in $\mb{P}^m_K$.
Let $Z=\ol{X}\bs X$.
Let $L=\mc{O}_{\ol{X}}(1)$ be the hyperplane line bundle on $\ol{X}$.
Since the statement of Theorem  \ref{Sparsity, abstract VMHS} is only about smooth variety $X$, we can assume that $\ol{X}$ is smooth by resolution of singularities \cite{Hir64}.
By enlarging $S$ if necessary, we can choose a smooth $\mc{O}_{K,S}$-model $\mc{X}$ of $X$.

A \emph{normal cycle} on a normal variety $W$ over a field of characteristic 0 is a finite morphism $T\ra W$ which is birational onto its image, where $T$ is a geometrically irreducible normal variety.
For any complete variety $Q$ over $K$ with an ample line bundle $J$, we let 
$$\deg(Q, J):=\lim_{k\ra\infty} \frac{\dim_K\Gamma(Q, J^{\otimes k})}{k^{\dim Q}}.$$
For any $x\in X(\CC)$, the local systems $\mc{L}_{an}$ and $\mc{L}_{n,an}:=\mc{L}_{an}\otimes (\ZZ/p^n\ZZ)$ induce the monodromy representation
$\pi_1^{\topo}(X_{\CC}^{an},x)\ra \Aut \mc{L}_{an,x}$
 and the mod $p^n$ monodromy representation
$\pi_1^{\topo}(X_{\CC}^{an},x)\ra  \Aut \mc{L}_{n, an,x}$ respectively.
For any $x\in X(\ol{K})$, the \'{e}tale local system $\mc{L}_{n,\CC,\et}$ induces an \'{e}tale monodromy representation 
$\pi_1^{\et}(X_{\CC},x)\ra  \Aut \mc{L}_{n,\et,\CC,x}$.
By the assumption $(\mc{L}_{n,\et,\CC})_{an}\simeq \mc{L}_{n,an}$, we have a commutative diagram
\begin{center}
\begin{tikzcd}
\pi_1^{\topo}(X_{\CC}^{an},x)\arrow[r] \arrow[d]& \Aut \mc{L}_{n,an,x}\arrow[d]
\\\pi_1^{\et}(X_{\CC},x)\arrow[r] &\Aut \mc{L}_{n,\et,\CC,x}.
\end{tikzcd}
\end{center}

\begin{lemma}\label{unbound monodromy}
Let $V$ be a complex irreducible subvariety of $\ol{X}_{\CC}$ such that $V$ is not contained in $Z_{\CC}$ and that $V^{\circ}:=V\cap X_{\CC}$ is not contained in a fiber of $\Phi$. 
Let $V^{\circ,s}$ be the smooth locus of $V^{\circ}$.
Then 
$\pi_1^{\topo}(V^{\circ,s})\ra \Aut \mc{L}_{an,x}$ has infinite image.  
\end{lemma}

\begin{proof}
Suppose $\pi_1^{\topo}(V^{\circ,s})\ra \Aut \mc{L}_{an,x}$ has finite image. 
Restrict the VMHS on $V^{\circ,s}$.  Passing to a finite cover $\wt{V^{\circ,s}}$ of $V^{\circ,s}$,
we get a VMHS on $V^{\circ,s}$ with trivial monodromy.
By rigidity \cite[Theorem 7.12]{BZ}, this VMHS is trivial, i.e. $\wt{\Phi}(V^{\circ,s})$ is a point, so $\Phi(V^{\circ})$ is a point, which contradicts that $V^{\circ}$ is not contained in a fiber of $\Phi$.
\end{proof}

\begin{lemma}\label{Construction of cover}
 For any $D\geq 1$, there exist a finite group $G$, a finite morphism $\tau: \ol{X}'\ra  \ol{X}$ of $K$-varieties, and an embedding $G\hookrightarrow \Aut (\ol{X}'/\ol{X})$ such that 
\begin{itemize}
\item $\tau|_{X}$ is finite \'{e}tale Galois with deck group $G$ and it extends to a finite \'{e}tale cover of the smooth integral model $\mc{X}$.
\item Let $U$ be an irreducible closed complex subvariety of $\ol{X}_{\CC}$. Let $U^{\circ}:=U\cap X_{\CC}$. Suppose  $U$ is not contained in $Z_{\CC}$ and $U^{\circ, an}$ is not contained in a single fiber of $\Phi$.
Let $Q$ be any irreducible component of $\tau^{-1}U$, endowed with the reduced structure, such that the induced finite map $\tau|_Q:Q\ra U$ is dominant. Then degree $\deg (Q, \tau^*L|_Q)\geq D$.
\end{itemize}
\end{lemma}

\begin{proof}
By \cite[Lemma 2.9]{BM22}, up to conjugation, there are only finitely many subgroups of 
$\pi^{\et}_1(X_{\CC})$ obtained as the image of 
$\pi_1^{\et}(f^{-1}(X_{\CC}))\ra \pi_1^{\et}(X_{\CC})$ 
with $f: T\ra \ol{X}_{\CC}$ a normal cycle such that 
\begin{itemize}
\item $\deg (T, f^*L_{\CC})\leq  D$, 
\item $f(T)$ is not contained in $Z_{\CC}$,
\item $f(T)^{\circ}:=f(T)\cap X_{\CC}$ is not contained in a single fiber of $\Phi$. 
\end{itemize} 
Let $E_1,\dots, E_r$ be such subgroups.
For any $i=1,\dots, r$, let $f_i:T_i\ra \ol{X}_{\CC}$ be normal cycles that induce these $E_i$ and have the listed properties above.
Let $F_i$ be the \'{e}tale fundamental group of the smooth locus $f_i(T_i)^{\circ,s}$ of $f_i(T_i)^{\circ}$.
Let $M_{i,n}$ be the image of $F_i$ under the \'{e}tale monodromy representation 
$\pi_1^{\et}(X_{\CC}, x)\ra \Aut \mc{L}_{n,\et,\CC, x}$. 
Let $F_i^{\topo}$ be the topological fundamental group of $f_i(T_i)^{\circ,s}$.
For any $i,n$, let $M_i^{\topo}$ and $M_{i,n}^{\topo}$ be the image of $F_i^{\topo}$ under the monodromy representation $\pi_1^{\topo}(X_{\CC}^{an},x)\ra \Aut \mc{L}_{an,x}$ and the mod $p^n$ representation respectively.
By Lemma \ref{unbound monodromy}, $M^{\topo}_i$ are infinite for all $i$.
Hence, the cardinalities of $M^{\topo}_{1,n},\dots, M^{\topo}_{r,n}$ can be made arbitrarily large uniformly (here we are using the finiteness) when $n\ra\infty$.
It follows from the commutative diagram in the beginning of this section that the same is true for the cardinalities of $M_{1,n},\dots, M_{r,n}$.

By constructibility, $X$ is open in $\ol{X}$, so $f_i(T_i)^{\circ,s}$ is open in $f_i(T_i)$.
Since $T_i$ is irreducible, $f_i(T_i)^{\circ,s}$ is irreducible.
Since $f_i^{-1}(f_i(T_i)^{\circ,s})$ is open in $T_i$, it is irreducible and normal.
Then by \cite[Lemma 11]{Kol03}, the image under $\pi^{et}_1(f_i^{-1}(f_i(T_i)^{\circ,s}))\ra F_i$ has finite index in $F_i$. 

Therefore, by fixing a large $n$, the cardinalities of the images of $E_1,\dots, E_r$ under the mod $p^n$ \'{e}tale monodromy representations can be made arbitrarily large uniformly, say $\geq (\dim X)! \cdot D$.

The \'{e}tale local system $\mc{L}_{n,\et}$ induces a finite \'{e}tale Galois cover of $X$ with deck group denoted by $G$. By enlarging $S$ if necessary, this cover extends to a finite \'{e}tale cover of the smooth integral model $\mc{X}$. Let $\tau:\ol{X}'\ra \ol{X}$ be normalization of $\ol{X}$ in this cover. The $G$-action on the cover extends uniquely to a $G$-action on $\tau$.

Let $\deg (\tau|_Q)$ be the degree of the finite map $\tau|_Q:Q\ra U$.
Suppose $\deg (U, L|_U)<D$. 
Let $\nu: U' \ra U$ be the normalization.
By the projection formula, 
$$\deg (U', \nu^*L|_{U'})=\deg (U, L|_U) <D.$$
The image of $\pi_1^{\et}(\nu^{-1}(U^{\circ}))\ra \pi_1^{\et}(U^{\circ})$ is conjugated to some $E_i$. 
By \cite[Lemma 2.10]{BM22}, $\deg (\tau|_Q)$ is equal to the cardinality of the image of the homomorphism 
$\pi_1^{\et}(\nu^{-1}(U^{\circ}))\ra G$. This cardinality is equal to the cardinality of the image of $\pi_1^{\et}(E_i)\ra G$, 
while this cardinality is $\geq (\dim X)!\cdot D$.
By asymptotic Riemann-Roch,
$$\deg (U, L|_U) :=\lim_{k\ra\infty} \frac{\dim_K\Gamma(U, L|_U^{\otimes k})}{k^{\dim U}}=\frac{(L|_U)^{\dim U}}{(\dim U)!}.$$
The intersection number $(L|_U)^{\dim U}$ is a positive integer. 
By the projection formula, 
$$\deg  (Q, \tau^*L|_Q)=\deg (\tau|_Q) \deg (U, L|_U).$$
Therefore, 
$$\deg  (Q, \tau^*L|_Q)\geq (\dim X)!\cdot D\cdot \frac{(L|_U)^{\dim U}}{(\dim U)!}\geq D.$$
In the case where $\deg (U,L|_U)\geq D$, we also have $\deg  (Q, \tau^*L|_Q)\geq D$.
\end{proof}

The remaining proof of Theorem \ref{Sparsity, abstract VMHS} is the same as the proof of \cite[Lemma 4.2]{ELV23} and Section 4.2 in \emph{op. cit.}. For completeness, we will give a sketch of it.

Let $\ell, d\geq 1$.
Let $V$ be a geometrically irreducible closed subvariety of $\ol{X}$ over $K$ of dimension $\ell$ and degree $d$ such that $V_{\CC}$ is not contained in $Z_{\CC}$ and $(V_{\CC}\cap X_{\CC})^{an}$ is not contained in a fiber of $\Phi$.

Choose $D$ such that $(\ell+1)/D^{1/\ell}<\varepsilon$. Taking this $D$ in Lemma \ref{Construction of cover}, we obtain a finite group $G$ and a finite morphism $\tau:\ol{X}'\ra \ol{X}$ satisfying the two properties therein.  
Let $X':=\tau^{-1}(X)$.

Firstly, there are finitely many covers $X'_j\ra X$ such that every $x\in X(\mc{O}_{K,S})$ lifts to a rational point in one of such covers:  A family of covers that satisfies this lifting property can be obtained by twisting the cover $X'\ra X$ and using \cite[Theorem 8.4.1]{P17}. Finiteness of such twists is due to Hermite-Minkowski theorem, see lines 9-21 of the proof of Lemma 4.2 of \cite{ELV23} for details. Let $\tau_j: \ol{X}'_j\ra \ol{X}$ be the normalization of $\ol{X}$ in the cover $X'_j\ra X$.

For a large enough integer $e$, the pullback $(\tau_j^* L)^{\otimes e}$ is very ample for all $j$. Use these line bundles to get projective embeddings $\ol{X}'_j\hookrightarrow \mb{P}^{M_j}$. There exists $c_{d,\varepsilon}>0$ such that for any integral point of $X\cap V$ with height $\leq B$, it is of the form $\tau_j(P)$ for some $P\in X'_j(K)\cap \tau_j^{-1}(V)(K)$ of height $\leq c_{d,\varepsilon}B^e$, see lines 22-34 of the proof of Lemma 4.2 of \cite{ELV23} for details.

Let $V^s$ be the smooth locus of $V$.
Let $V^{s,\circ}:= V^s\cap X$.
Since $\tau_j^{-1}(V^{s,\circ})$ is a finite \'{e}tale cover of the geometrically irreducible smooth $K$-variety $V^{s,\circ}$, its geometric components are pairwise distinct by \cite[Exp. I., Cor. 10.8]{GR71} and permuted by $\Gal(\ol{K}/K)$. 
The geometric components of $\tau_j^{-1}(V^{s,\circ})$ having a $K$-rational point are thus defined over $K$, and the number of such components is bounded by the cardinality of $G$. Let $Q^{\circ}$ be one of such components. The $K$-Zariski closure $Q$ of $Q^{\circ}$ is geometrically irreducible. The map $\tau_j:X'_j\ra X$ induces a map $\tau_j: Q\ra V$. Since $\tau_j$ is \'{e}tale over $V^{s,\circ}$, the image $\tau_j(Q)$ contains an open subset, so $\tau_j: Q\ra V$ is dominant. 
By the second property of Lemma \ref{Construction of cover} and the fact that $\tau_j$ is the twist of $\tau$, 
$\deg (Q, \tau^*L|_Q)\geq D$.
Then as in the last three paragraphs of Lemma 4.2 of \cite{ELV23} (with the only difference that $\varepsilon$ is rescaled to $\varepsilon/2e$ instead of $\varepsilon/2$ at the very end because we were taking a slightly different approach to bound the degree of $Q$; also note that $e$ is independent of $B$ and $V$, and can be chosen depending only on $\varepsilon$ and $\ell$),  using  Broberg's theorem \cite{Bro04} (which builds on fundamental ideas of Bombieri-Pila \cite{BP89} and Heath-Brown \cite{Hea02}), we can obtain the following lemma:

\begin{lemma}\label{covering subvariety count, codimension one}
Let $V$ be a geometrically irreducible closed subvariety of $\ol{X}$ over $K$ of dimension $\ell$ and degree $d$ such that $V_{\CC}$ is not contained in $Z_{\CC}$ and $(V_{\CC}\cap X_{\CC})^{an}$ is not contained in a fiber of $\Phi$. Then all integral points of $X\cap V$ of height $\leq B$ can be covered by $O_{d,\varepsilon, \ell}(B^{\varepsilon})$ irreducible subvariety over $K$ of dimension $\leq \ell-1$ and degree $O_{d,\varepsilon}(1)$.
\end{lemma}

For the irreducible subvarieties obtained in Lemma \ref{covering subvariety count, codimension one} that are not geometrically irreducible, their rational points can be covered by $O_{d,\varepsilon,\ell}(1)$ subvarieties of smaller dimensions and of degree $O_{d,\varepsilon, \ell}(1)$ using \cite[Lemma 2.4(c)]{ELV23}. For the geometrically irreducible ones that cover the integral points but does not contained in $Z_{\CC}$ and not contained in a fiber of $\Phi$, we can apply Lemma \ref{covering subvariety count, codimension one} again on them. By starting from $X$ instead and repeating this procedure, we can deduce Theorem \ref{Sparsity, abstract VMHS}, as in  \cite[Section 4.2]{ELV23}.

\subsection{Proof of Theorem \ref{Sparsity, surjective quasi-projective}}
By Saito's theory \cite[Equation 2.18.1 and Theorem 3.27]{Sai90},
$R^i\pi_! \QQ$ underlies a mixed Hodge module, which is an admissible graded-polarized variation of mixed $\QQ$-Hodge structures on $(X_1)_{\CC}$, where $X_1$ is the non-empty Zariski open subset of $X$ on which the perverse sheaf $R^i\pi_! \QQ$ is a local system.
Let $p$ be a prime for which $\mc{L}_{an}:=R^i\pi_! \ZZ|_{X_1}$ is $p$-torsion-free. 
By \cite[Theorem 12.10 and Theorem 12.15]{FK}, there exists a non-empty open subset $X_2$ of $X$ such that $R^i\pi_!\ZZ_p|_{X_2}$ is a lisse $p$-adic sheaf.
Take $X^*$ to be $X_1\cap X_2$.
We have an isomorphism $(R^i\pi_! \ZZ_p|_{X^*})_{\CC}^{an}\simeq \mc{L}_{an}|_{X^*}\otimes \ZZ_p$ of $\ZZ_p$-local systems.
Theorem \ref{Sparsity, surjective quasi-projective} then follows from Theorem \ref{Sparsity, abstract VMHS}.


\begin{bibdiv}
\begin{biblist}


\bib{Bat93}{article} {author={V. V. Batyrev}, title={Variations of the mixed Hodge structure of affine hypersurfaces in algebraic tori},  journal={Duke Math. J.}, volume={69(2)}, date={1993},  pages={349--409}}

\bib{BP89}{article}{author={E. Bombieri}, author={J. Pila}, title={The number of integral points on arcs and ovals}, journal={Duke Math. J.}, volume={59(2)}, date={1989},  pages={337--357}}

\bib{Bro04}{article}{author={N. Broberg}, title={A note on a paper by R. Heath-Brown: ``The density of rational points on curves and surfaces''}, journal={J. reine angew. Math.}, volume={571}, date={2004},  pages={159--178}}

\bib{BE}{book} {author={P. Brosnan}, author={F. El Zein}, title={Variation Of Mixed Hodge Structures},  publisher={in Hodge theory (edited by E. Cattani, F. El Zein, P. A. Griffiths, D. T. L\^{e}), Princeton University Press},  date={2014}}

\bib{BM22}{article} {author={Y. Brunebarbe}, author={M. Maculan},  title={Counting integral points of bounded height on varieties with large fundamental group}, journal={J. Reine Angew. Math.}, volume={807}, date={2024}, pages={31--53}}

\bib{BZ}{book} {author={J.-L. Brylinski}, author={S. Zucker}, title={An overview of recent advances in Hodge theory},  publisher={in Complex Manifolds, Springer}, date={1998}}

\bib{ELV23}{article} {author={J. S. Ellenberg}, author={B. Lawrence}, author={A. Venkatesh}, title={Sparsity of Integral Points on Moduli Spaces of Varieties},  journal={Int. Math. Res. Not. IMRN}, volume={17}, date={2023},  pages={15073--15101}}

\bib{Fal83}{article} {author={G. Faltings}, title={Endlichkeitss\"{a}tze f\"{u}r abelsche Variet\"{a}ten \"{u}ber Zahlk\"{o}rpern},  journal={Invent. Math.}, volume={73(3)}, date={1983},  pages={349--366}}

\bib{FK}{book} {author={E. Freitag}, author={R. Kiehl}, title={Etale cohomology and the Weil conjecture},  publisher ={Springer-Verlag}, date={1988}}

\bib{FF}{article} {author={O. Fujino}, author={T. Fujisawa}, title={Variations of mixed Hodge structure and semi-positivity theorems},  journal={Publ. RIMS Kyoto Univ.}, volume={50}, date={2014},  pages={589--661}}

\bib{GKZ94}{book} {author={I. M. Gelfand}, author={M.M. Kapranov}, author={A.V. Zelevinsky}, title={Discriminants, Resultants, and Multidimensional Determinants}, publisher={Birkh\"{a}user}, date={1994}}

\bib{GR71}{book} {author={A. Grothendieck}, author={M. Raynaud},  title={Rev\^{e}tements \'{E}tales et Groupe Fondamental (SGAI),  S\'{e}minaire de g\'{e}om\'{e}trie alg\'{e}brique du Bois Marie,1960/61}, publisher={in Lecture notes in mathematics 224, Springer-Verlag}, date={1971}}

\bib{Hea02}{article}{author={D. R. Heath-Brown}, title={The Density of Rational Points on Curves and Surfaces}, journal={Ann. of Math.}, volume={155(2)}, date={2002},  pages={553--598}}

\bib{Hir64}{article} {author={H. Hironaka},  title={Resolution of Singularities of an Algebraic Variety Over a Field of Characteristic Zero: I}, journal={Ann. of Math.}, volume={79(1)}, date={1964},  pages={109--203}}

\bib{Kol03}{book} {author={J. Koll\'{a}r},  title={Rationally Connected Varieties and Fundamental Groups}, publisher={in Higher Dimensional Varieties and Rational Points, p. 69--92, Royal Society Mathematical Studies 12, Springer}, date={2003}}

\bib{LS20}{article} {author={B. Lawrence}, author={W. Sawin}, title={The Shafarevich conjecture for hypersurfaces in abelian varieties}, pages={arXiv:2004.09046v2}}

\bib{LV20}{article} {author={B. Lawrence}, author={A. Venkatesh}, title={Diophantine problems and $p$-adic period mappings},  journal={Invent. Math.}, volume={221(3)}, date={2020},   pages={893--999}}

\bib{PS}{book} {author={C. A. M. Peters},author={J. H. M. Steenbrink},  title={Mixed Hodge Structures}, publisher={Springer}, date={2008}}

\bib{P17}{book} {author={B. Poonen},  title={Rational Points on Varieties}, publisher={Graduate Studies in Mathematics 186, American Mathematical Society}, date={2017}}

\bib{Sai90}{article} {author={M. Saito}, title={Mixed Hodge Modules},  journal={Publ. RIMS, Kyoto Univ.}, volume={26}, date={1990},   pages={221--333}}


\end{biblist}
\end{bibdiv}
\end{document}